\documentclass[11pt]{amsart}
\usepackage{latexsym, amsmath, amssymb,amsthm,amsopn,amsfonts}
\usepackage{version}
\usepackage{epsfig,graphics,color,graphicx,graphpap}
\usepackage{amssymb}
\usepackage{todonotes}
\usepackage{mathrsfs}
\usepackage[citecolor=blue]{hyperref}
\usepackage{esint}
 
\newcommand{\redsout}{\bgroup\markoverwith{\textcolor{red}{\rule[0.5ex]{2pt}{.4pt}}}\ULon}

\newcommand{\p}{\partial}
\newcommand{\Om}{\Omega}
\newcommand{\e}{\varepsilon}
\newcommand{\grad}{\nabla}
\newcommand{\Kn}{\mathsf{Kn}}

\numberwithin{equation}{section}
\newtheorem{theorem}{Theorem}[section]

\newtheorem{proposition}[theorem]{Proposition}
\newtheorem{lemma}[theorem]{Lemma}

\newtheorem{remark}{Remark}[section]
 
\newcommand{\R}{\mathbb R}

\graphicspath{{../Figures/}}

\title[ ]{On diffusive scaling in acousto-optic imaging}

\author[Chung]{Francis J. Chung}
\address{Department of Mathematics, University of Kentucky, Lexington, KY 40506, USA}
\email{fj.chung@uky.edu}

\author[Lai]{Ru-Yu Lai}
\address{School of Mathematics, University of Minnesota, Minneapolis, MN 55455, USA}
\curraddr{}
\email{rylai@umn.edu}

\author[Li]{Qin Li}
\address{Department of Mathematics and Wisconsin Institute for Discoveries, University of Wisconsin-Madison, Madison, WI 53705, USA}
\curraddr{}
\email{qinli@math.wisc.edu}
\thanks{Chung is supported in part by Simons Collaboration Grant 582020. Lai is supported in part by the NSF grant DMS-1714490. Li is supported in part by DMS-1619778 and DMS-1750488. The collaboration started during the workshop ``Mathematics in Optical Imaging" in IMA, 2019, and was continuously supported by KI-net, DMS-1107291 in the workshop ``Forward and Inverse Problems in Kinetic Theory" in UW-Madison, 2019. The authors thank Professor John Schotland for helpful discussions.}
\thanks{\textit{Key words and phrases:} Acousto-optic imaging, inverse problem, radiative transfer equation.}

\date{}

\begin{document}
\maketitle

\begin{abstract}
Acousto-optic imaging (AOI) is a hybrid imaging process. By perturbing the to-be-reconstructed tissues with acoustic waves, one introduces the interaction between the acoustic and optical waves, leading to a more stable reconstruction of the optical properties. The mathematical model was described in~\cite{ChungSchotland}, with the radiative transfer equation serving as the forward model for the optical transport. In this paper we investigate the stability of the reconstruction. In particular, we are interested in how the stability depends on the Knudsen number, $\Kn$, a quantity that measures the intensity of the scattering effect of photon particles in a media. Our analysis shows that as $\Kn$ decreases to zero, photons scatter more frequently, and since information is lost, the reconstruction becomes harder. To counter this effect, devices need to be constructed so that laser beam is highly concentrated. We will give a quantitative error bound, and explicitly show that such concentration has an exponential dependence on $\Kn$. Numerical evidence will be provided to verify the proof.
\end{abstract}

\section{Introduction}
High energy light is the classical way to probe optical properties of thick and highly scattering media. In contrast to most destructive experiments, here, high energy photons are injected into biological tissues, and the outgoing light intensities are measured at the surface of the samples by detectors. The map that maps the incoming light intensity to the outgoing light intensity carries optical information of the biological tissues, and is utilized to reconstruct optical parameters.

However, a long standing challenge of the reconstruction process centers around the stability issue. If the injected photons carry low energy, the resulting images are very blurred, and the reconstruction is typically unstable~\cite{Bal09,Bal10a,CLW, Cheng_Gamba_Ren_Doping, LLU2018,Hielscher_rheumatoid2}. This phenomenon is the so-called diffuse optical tomography, especially when the injected light is near-infrared. In this situation, the forward model is the classical elliptic type equation, and mathematically, recovering the optical parameter amounts to reconstructing the diffusion coefficient in the elliptic equation using the Dirichlet-to-Neumann map~\cite{uhlmann2009electrical}, and is proven mathematically to be logarithmically unstable~\cite{A1}. Multiple strategies are adopted to ``stabilize" the problem~\cite{Ren_fPAT,Arridge_couple}, both by adjusting the experimental modalities upfront, or introducing image deblurring techniques as a post-processing~\cite{Chen_2018,jin2019convergence,Arridge_diffusion, Arridge_diffusion_error}.

We follow the first approach in this paper. In particular, we investigate a modality called acousto-optic imaging (AOI), and study the stability of the media reconstruction when biological tissue is optically thick. AOI is one of the  state-of-the-art hybrid imaging processes that fuse two or more imaging modalities to form a new imaging technique \cite{Bal2012review, Kuchment}. In particular, AOI is based on the acousto-optic effect: optical properties of the medium are modified upon interaction with acoustic radiation. Presumably one can obtain more information upon a sequence of such modifications. The procedure combines the contrast advantage inheriting from optical properties and the resolution advantage of ultrasound~ \cite{BalSchotland2010,Lesaffrea,LVWang}, and is expected to provide a more stable reconstruction.

In general, there are two types of AOI - direct imaging and tomographic imaging. In direct imaging, the sample is simultaneously affected by an ultrasound beam and illuminated with a laser source. The tagged photon, resulting from the interaction between the light and the ultrasound, is detected by external ultrasound transducers and then forms the image \cite{Elson2011}. Tomographic imaging, the second type of AOI, makes use of a reconstruction algorithm to track down the optical properties of the medium through various inversion procedures. In the latter case, the inverse problems for incoherent AOI \cite{AmmariBGNS,AmmariGNS, AmmariNS, BalMoskow2014, BalSchotland2010, BalSchotland2014} and for coherent AOI with diffuse light \cite{ChungJS} were studied and the absorption and the diffusion coefficients of a scattering medium are stably reconstructed. On the other hand, for the problem within the framework of radiative transport theory, relevant results are also investigated in \cite{BalChungSchotland, CHS2019, ChungSchotland}; also see reviews~\cite{Arridge99,Arridge_Schotland09,Bal2012review,Kuireview}.

As a typical hybrid imaging problem, the inverse problem in acousto-optic imaging involves two steps: the first step amounts to reconstructing internal information from triggering the system using ultrasound. Ultrasound waves of different frequencies imposed on a physical system induces a wave-like perturbation to the media, leading to wave-like perturbation in the measurement. From measurements of this perturbation over a big spectrum of frequencies, a Fourier-type calculation provides internal information inside the physical domain. The second step then amounts to reconstructing optical parameters in the modeling equation using the internal data. The first step basically is an inverse Fourier series, and thus is regarded as a stable process, and since the internal data is obtained, it is largely believed that the reconstruction is also stable in the second step. 

This is indeed the situation that we will be studying. As a forward model we use the radiative transfer equation (RTE) to characterize the dynamics of photon particles. This equation, besides containing optical parameters, also depends on the Knudsen number $\Kn$: it is the ratio of mean free path and the typical domain length. In the thick and dense optical environment, the photon particles scatter frequently, and mathematically this amounts to setting this ratio to be small. One can also view this quantity to reflect the energy of the photon particles being used: indeed, for the same media, in the low energy regime, photon particles are scattered more often. One interesting phenomenon associated with low energy, high scattering situation is that, in this regime, mathematically one can asymptotically connects the RTE to a parabolic heat equation. It is a well-accepted fact that reconstruction of parameters are unstable for diffusion type equations, and thus it is a natural expectation that the inverse problem of the RTE is ill-posed accordingly in this regime.

The rest of this paper is organized as follows. In Section~\ref{Sec:main}, we discuss the first step in AOI of determining the internal data, introduce preliminary results for the well-posedness of the transport equation, and present the main result in Theorem \ref{thm:main}. Section~\ref{sec:reconstruction} is concerned with the second step in AOI. We derive the reconstruction formula for the absorption coefficient from the reconstructed internal data and focus on the dependence of $\Kn$ in the associated convergence rate. The numerical evidence is presented in Section~\ref{sec:numeric}.

\section{Problem setup and main results}\label{Sec:main}
In this section, we start by introducing some preliminary results and then describe the setup of the problem. Lastly, we present the main theorem. 

To begin, consider a smooth bounded domain $\Omega \subset \mathbb{R}^n$. In this paper we use the RTE as the forward model for light propagation on $\Omega$. We also assume the time scale is significantly larger, and the equation achieves the steady state:
\begin{align}\label{equation RTE}
\left\{\begin{array}{ll}
\theta\cdot\nabla u = \frac{\sigma_s(x)}{\Kn}\left(\langle u\rangle - u\right) -\Kn\sigma_a(x) u& \hbox{in }\Omega\times\mathbb{S}^{n-1}\,,\\
u|_{\Gamma_-}=f\,.\\
\end{array}\right.
\end{align}
The solution $u(x,\theta)$ characterizes the density of photon particles at a physical point $x \in \Omega$ moving in direction $\theta\in\mathbb{S}^{n-1}$ for $n\geq 2$. The left-hand side of \eqref{equation RTE} is a transport term describing particles at the position $x$ moving in the $\theta$ direction. The terms on the right represent the interaction of the photon particles with the medium, where $\sigma_a$ is the absorption coefficient, $\sigma_s$ is the scattering coefficient. The scattering operator is $\langle u\rangle - u$ with
\[
\langle u \rangle := \frac{1}{|\mathbb{S}^{n-1}|}\int_{\mathbb{S}^{n-1}} u(x,\theta)~d\theta =: \fint u(x,\theta)~d\theta
\]
standing for the angular average of $u$. Physically it means the particles with different velocities at $x$ are redistributed in an isotropic way at rate $\sigma_s(x)$. It is also common to insert an anisotropic term $k(x,\theta,\theta')$ in the integral defined above, but it does not bring extra mathematical difficulty and thus will be neglected here. The parameter $\Kn$ in the equation is called the Knudsen number, and as described above, represents the ratio of mean free path to the typical domain length.

For conciseness of notation, we denote the total absorption coefficient
\[
\sigma=\sigma_s+\Kn^2\sigma_a\,,
\]
and rewrite the right-hand side of \eqref{equation RTE} as a linear operator 
\begin{equation}\label{eqn:def_L}
\mathcal{L}u:=\frac{\sigma_s }{\Kn} \langle u\rangle - {\frac{\sigma}{\Kn}} u\,.
\end{equation}
The boundary condition of \eqref{equation RTE} is imposed on $\Gamma_-$: a a set of coordinates at the physical boundary $\partial\Omega$ with inward velocities. More specifically:
\begin{align}\label{gamma+}
\Gamma_{\pm}:=\{(x,\theta)\in \p \Omega\times \R^3: \pm \ n(x)\cdot \theta > 0\}\,, 
\end{align}
where $n(x)$ is the unit outer normal to $\p\Omega$ at the point $x\in \p\Omega$. Similarly, $\Gamma_+$ collects all outgoing coordinates at the boundary.

The well-posedness problem of \eqref{equation RTE} has been extensively studied. We refer the readers to, for instance, \cite{CS2, CS3, CS98} for detailed discussion on the existence of solutions to \eqref{equation RTE} under different constraints. For our purpose, we will be using the following boundedness result, with $\Kn$ included and adjusted to fit our setting.

\begin{theorem}[Theorem 2.1 in \cite{BalChungSchotland}]\label{RTESolutionEstimate} 
Suppose $\sigma_s$ and $\sigma_a$ are bounded and uniformly positive functions in $\Omega$ and the source term $S\in L^{\infty}(\Om \times \mathbb{S}^{n-1})$. For any $\Kn > 0$ and any boundary value $f \in L^{\infty}(\Gamma_-)$, there exists a strong unique solution $u \in L^{\infty}(\Om \times \mathbb{S}^{n-1})$ to the boundary value problem
\begin{align}\label{SourceRTE}
\left\{\begin{array}{ll}
\theta\cdot\nabla u  = {\frac{\sigma_s}{\Kn}}( \langle u\rangle - u) - \Kn \sigma_a u  + S & \hbox{in }\Omega\times \mathbb{S}^{n-1}\,,\\
u  =f & \hbox{on }\Gamma_-\,.\\
\end{array}\right.
\end{align}
Furthermore we have the estimate:
\begin{equation}\label{SolutionEstimate}
\|u\|_{L^{\infty}(\Om \times \mathbb{S}^{n-1})} \leq \frac{C}{\Kn}\left(\|S\|_{L^{\infty}(\Om \times \mathbb{S}^{n-1})}  + \frac{1}{\Kn}\|f\|_{L^{\infty}(\Gamma_-)}\right),
\end{equation}
holds for some positive constant $C$ depending only on $\sigma_a$ and $\sigma_s$. 
\end{theorem}
In the low energy regime, $\Kn$ is small, and one can perform the asymptotic expansion in $\Kn$. In the leading order the scattering term $\frac{\sigma_s}{ \Kn}( \langle u\rangle - u)$ becomes dominant, and it drives $u\sim\langle u\rangle$, meaning $u$ is a constant in $\theta$ at every location $x$, see~\cite{CLW}. This degeneracy of information, as will be explained in details in Section 3, is one of the main reason for the rise of instability.

With the well-posedness result in Theorem~\ref{RTESolutionEstimate}, the trace can be taken and we define the albedo operator $\mathcal{A}$:
\begin{align}\label{definition albedo}
\mathcal{A}: L^{\infty}(\Gamma_-)\rightarrow L^\infty(\Gamma_+):\quad u|_{\Gamma_+}  =\mathcal{A}f\,.
\end{align}
It maps the incoming light intensity to the outgoing light intensity.

\subsection{Problem setup}
Acousto-optic imaging is an imaging modality that introduces an acoustic wave to the medium, and thus induces a wave-like perturbation to the coefficients $\sigma_a$ and $\sigma_s$, that gives rise to new coefficients $\sigma_{a,\varepsilon}$ and $\sigma_{s,\varepsilon}$. Following the setup presented in \cite{ChungSchotland}, the modulated coefficients take the form
\begin{align}\label{perturbed absorption}
\sigma_{a,\varepsilon} := (1+\varepsilon \cos(q\cdot x))\sigma_a\,,\quad\sigma_{s,\varepsilon} := (1+\varepsilon \cos(q\cdot x))\sigma_s\,,
\end{align}
where $0<\varepsilon\ll 1$ is the strength of introduced acoustic effect. The $\cos{(q\cdot x)}$ perturbation comes from the assumption that the introduced acoustic wave is basically sinusoidal, and the wave vector $q$ can be experimentally adjusted.

We denote $u_\varepsilon$ to be the solution to the RTE with the modulated coefficients, namely:
\begin{align}\label{equation modulated RTE}
\begin{cases}
\theta\cdot\nabla u_{\varepsilon} = \frac{\sigma_{s,\varepsilon}}{\Kn}( \langle u_{\varepsilon} \rangle - u_{\varepsilon}) - \Kn \sigma_{a,\varepsilon} u_{\varepsilon}\,,\\
u_{\varepsilon}|_{\Gamma_-}=f\,.\\
\end{cases}
\end{align}
For small $\varepsilon$, the well-posedness result for \eqref{equation modulated RTE} still holds true according to Theorem~\ref{RTESolutionEstimate}. Hence, for each $q$, we can define the corresponding albedo operator:
\[
    \mathcal{A}_{\varepsilon,q}: f\mapsto u_\varepsilon|_{\Gamma_+}\,,
\]
where $u_\varepsilon$ is the solution to \eqref{equation modulated RTE}. Likewise, we also define the perturbed scattering operator $\mathcal{L}_\varepsilon$:
\[
\mathcal{L}_\varepsilon u :=\frac{\sigma_{s,\varepsilon}}{\Kn} \langle u  \rangle - \frac{\sigma_{\varepsilon}}{\Kn} u \,,
\]
where the modulated total absorption coefficient is
\[
\sigma_\varepsilon:=\sigma_{s,\varepsilon}+\Kn^2\sigma_{a,\varepsilon}\,.
\]

To extract internal information, we first introduce the adjoint RTE:
\begin{align}\label{adjoint equation}
\begin{cases}
-\theta\cdot \nabla v= \mathcal{L}v\,,\\
v|_{\Gamma_+}=g\,.\\
\end{cases} 
\end{align}
It is straightforward to show that if $v$ solves the adjoint RTE, then
\[
\tilde{v}(x,\theta) = v(x,-\theta)
\]
solves the regular RTE. Therefore the well-posedness result applies and we can also define the adjoint albedo operator
\[
\tilde{\mathcal{A}}: L^\infty(\Gamma_+) \rightarrow L^\infty(\Gamma_-): v|_{\Gamma_-} =\tilde{\mathcal{A}}g\,.
\]
Furthermore: 
\[
\widetilde{\mathcal{A}(\tilde{g})}=\tilde{\mathcal{A}}(g)\,,
\]
so $\tilde{\mathcal{A}}$ is determined by $\mathcal{A}$.
Here for a given function $\varphi$, the expression $\tilde{\varphi}$ indicates the reflection of $\varphi$ in the $\theta$ variable, namely, $\tilde{\varphi}(x,\theta):=\varphi(x,-\theta)$.

Now we can use the adjoint RTE to obtain an internal functional. Using ~\eqref{equation modulated RTE} and ~\eqref{adjoint equation} we find
\begin{align}\label{internal id}
  \int_{\partial\Omega\times\mathbb{S}^{n-1}} \theta\cdot n(x) (u_\varepsilon v) ~dxd\theta 
 &= \int_{\Omega}\int_{\mathbb{S}^{n-1}} (v\mathcal{L}_\varepsilon u_\varepsilon  - u_\varepsilon \mathcal{L}v ) ~dxd\theta\, \notag\\
 &=  \int_{\Omega}\int_{\mathbb{S}^{n-1}} v(\mathcal{L}_\varepsilon -\mathcal{L}) u_\varepsilon   ~dxd\theta\,,\notag\\
  &=  \int_{\Omega}\int_{\mathbb{S}^{n-1}} \varepsilon \cos(q\cdot x) (\mathcal{L} u_\varepsilon) v ~dxd\theta\,,
\end{align}
where we used the self-adjoint property of $\mathcal{L}$ and $\mathcal{L}_\varepsilon$ as well as $\mathcal{L}_\varepsilon -\mathcal{L}=\varepsilon \cos(q\cdot x) \mathcal{L}$. 
We can label the left-hand side of this $BT$ (for boundary term) and rewrite it in terms of boundary values of $v$ and $u_{\varepsilon}$:
$$
   BT:=\int_{\Gamma_-} \theta \cdot n f\tilde{\mathcal{A}}(g)~dS + \int_{\Gamma_+} \theta \cdot n \mathcal{A}_{\varepsilon,q}(f)g ~dS\,.
$$
Since the boundary values are known, this implies that $BT$ is known.

Finally, in \eqref{internal id}, we can write $\mathcal{L}u_{\varepsilon} = \mathcal{L}u + \mathcal{L}(u_{\varepsilon}-u)$, so if  $u_\varepsilon - u$ is small when $\varepsilon\to 0$, \eqref{internal id} becomes 
\begin{align}\label{integral identity A_approx}
\varepsilon |\mathbb{S}^{n-1}|\int_{\Omega} \cos(q\cdot x)H(x)~dx  +\mathcal{O}(\varepsilon^2) =BT\,,
\end{align}
where one defines the internal data
$H$ as
\begin{equation}\label{eqn:def_H}
H(x) := \fint_{\mathbb{S}^{n-1}} v\mathcal{L}u  ~d\theta\,.
\end{equation}
This means that in the leading order, with the measurement at the boundary, we can recover the quantity 
\begin{equation*}\label{eqn:BT_leading}
\varepsilon\int_{\Omega} \cos(q\cdot x)H(x)~dx \,.
\end{equation*}
As the reconstruction can be done for all $q\in \R^n$, we obtain the entire spectrum of the Fourier transform of $H$, which is further used to recover $H$.  \footnote{We assume that $\sigma_a,\sigma_s$ are compactly supported in $\Omega$.}

To justify~\eqref{integral identity A_approx}, we use the smallness of $u-u_\varepsilon$ described in the following Lemma:
\begin{lemma}\label{Oepsilon}
For small $\varepsilon$, we have
\[
\|u_{\varepsilon} - u\|_{L^{\infty}(\Omega \times \mathbb{S}^{n-1})} \leq {\mathcal{O}(\varepsilon)}\|u_{\varepsilon}\|_{L^{\infty}(\Om \times \mathbb{S}^{n-1})}\,.
\]
\end{lemma} 
\begin{proof}
Since $u$ solves the RTE $\theta \cdot \grad u = \mathcal{L} u$ and $u_{\varepsilon}$ solves the modulated RTE $\theta \cdot \grad u_{\varepsilon} = \mathcal{L}_{\varepsilon} u_{\varepsilon}$ with the same incoming boundary conditions, we can write
	\begin{align*}
		\theta \cdot \grad (u_{\varepsilon} - u) &= \mathcal{L}_{\varepsilon} u_{\varepsilon} - \mathcal{L}u \\
		 &=\mathcal{L}(u_{\varepsilon} - u) +( \mathcal{L}_{\varepsilon} - \mathcal{L})u_{\varepsilon} \\
		&= \mathcal{L}(u_{\varepsilon} - u) +\e \cos(q\cdot x) \mathcal{L}u_{\varepsilon} 
	\end{align*}
	and $u_{\varepsilon} - u = 0$ on $\Gamma_-$.   It follows from Theorem \ref{RTESolutionEstimate} that 
	\[
	\|u_{\varepsilon} - u\|_{L^{\infty}(\Omega \times \mathbb{S}^{n-1})} \leq {\mathcal{O}(\varepsilon)}\|\mathcal{L}u_{\varepsilon}\|_{L^{\infty}(\Om \times \mathbb{S}^{n-1})}\,,
	\]
	and thus from the definition of $\mathcal{L}$ we derive 
	\begin{equation*}
	\|u_{\varepsilon} - u\|_{L^{\infty}(\Omega \times \mathbb{S}^{n-1})} \leq {\mathcal{O}(\varepsilon)}\|u_{\varepsilon}\|_{L^{\infty}(\Om \times \mathbb{S}^{n-1})}\,.
	\end{equation*}
\end{proof}

\begin{remark}
We comment that the proof above may lead to a large constant in the $\mathcal{O}$ notation for small $\Kn$, but this constant will not affect the later proof. Furthermore, the bounds shown above is not sharp in the diffusion limit. In the diffusion limit of $\Kn\to 0$, $u$ and $u_\varepsilon$ approximately satisfy two elliptic equations. By comparing these two elliptic equations, we can obtain the smallness of $u-u_\varepsilon$ in $\varepsilon$, and such smallness will be independent of $\Kn$.
\end{remark}

This immediately leads to the following proposition:
\begin{proposition}\label{def H}
	Suppose that $f, \ g\in L^\infty(\Gamma_-)$. Let $u$ be the solution to the RTE with boundary data $u|_{\Gamma_-}=f$, and $v$ be the solution to the adjoint RTE with boundary data $v|_{\Gamma_+}=g$. 
	Then $\tilde{\mathcal{A}}(g)$ and $\mathcal{A}_{\varepsilon,q}(f)$ uniquely determine the internal functional $H\in L^\infty(\Omega)$ up to $\mathcal{O}(\varepsilon)$ accuracy:
	\begin{align}\label{internal data}
	H(x) = \frac{\sigma_s(x)}{\Kn} \langle u\rangle\langle v \rangle - \frac{\sigma(x)}{\Kn} \langle uv\rangle\,.
	\end{align} 
\end{proposition}
This directly comes from inserting the definition of $\mathcal{L}$~\eqref{eqn:def_L} in that of $H$ in~\eqref{eqn:def_H}.

\subsection{Main results}
The main objective of this paper is to analyze the stability of the reconstruction process of AOI. In particular we are interested in the requirement imposed on the probing light when the diffusion effect is strong. More specifically, we consider the incoming light $f$ that is concentrated in angle:

\begin{equation*}
f (x,\theta):=\left\{\begin{array}{ll} 
   c_nh^{\frac{1}{2} (1-n)}  & \hbox{if }|\theta-\theta_0|<h\,,\\
   0 & \hbox{otherwise}\,,\\
   \end{array}\right.
\end{equation*}
in $\Omega\times\mathbb{S}^{n-1}$, where $h$ stands for the concentration of the light beam, and $c_n$ is chosen to be the constant depending only on $n$ such that 
\begin{align}\label{identity 1}
    \fint_{\mathbb{S}^{n-1}} f^2 d\theta =1\,.
\end{align}

With concentrated light beam injected into the tissue, singularities are more likely to be preserved, keeping more information for the reconstruction. As $\Kn$ decreases, media becomes optically thicker, and the light beam needs to be more concentrated for a stable reconstruction. We will carefully trace such dependence of $h$ on $\Kn$.
 
The main theorem we will prove is the following, and we leave the proof to Section~\ref{sec:reconstruction}.

\begin{theorem}\label{thm:main}
Let $H$ be the internal functional defined in Proposition~\ref{def H}. Suppose that $f$ satisfies \eqref{eqn:boundary}. If $\sigma \in C^{\alpha}(\Om)$, for some $0<\alpha \leq 1$, then for all $x \in \Om$ one has
\begin{equation}\label{PreSigmaFormula}
\frac{\sigma(x)}{\Kn} = -\frac{f(x - \tau_-(x,\theta_0)\theta_0,\theta_0)}{\mathcal{A}(f)(x + \tau_+(x,\theta_0)\theta_0,\theta_0)} H(x) + \mathcal{O}(h^{\alpha})e^{\frac{3d \sup \sigma}{\Kn}}\,,
\end{equation}
where $\tau_{\pm}$ represent the distance from $x$ to $\partial\Omega$ in the direction $\pm \theta$, and $d$ is the diameter of the domain $\Omega$. The constant implied in the big $\mathcal{O}$ is independent of $\Kn$. Moreover, if one chooses
\begin{equation}\label{HCondition}
h \ll e^{\frac{-3d \sup \sigma}{\alpha \Kn}}\,,
\end{equation}
then 
\begin{equation}\label{SigmaFormula}
\frac{\sigma(x)}{\Kn} = -\frac{f(x - \tau_-(x,\theta_0)\theta_0,\theta_0)}{\mathcal{A}(f)(x + \tau_+(x,\theta_0)\theta_0,\theta_0)}  H(x) + \mathcal{O}(h^{\alpha'})\,,
\end{equation}
for some $0 < \alpha' < \alpha$. 
\end{theorem}

We make the following five remarks regarding the interpretation of Theorem \ref{thm:main}.

First, on explicit reconstruction: we note, as in \cite{ChungSchotland}, that everything on the right side of \eqref{SigmaFormula}, other than the error term, is explicit: can be directly calculated from boundary measurements.  Thus \eqref{SigmaFormula} can be viewed as a reconstruction formula for $\sigma$ in terms of boundary quantities.

Second, in the case $\Kn \rightarrow 0$, we have $\Kn^2 \sigma_a \rightarrow 0$ in the total absorption coefficient $\sigma$.  Therefore it makes sense in this case to consider the reconstruction of one coefficient by \eqref{SigmaFormula}, and not attempt to recover $\sigma_a$ and $\sigma_s$ separately as in \cite{ChungSchotland}.

Third, one preliminary assumption of $\sigma$: the theorem does require H\"older continuity. This is different from what is required in~\cite{ChungSchotland} where H\"older continuity is removed. The main difference is we are establishing a \emph{rate of convergence} for the formula in terms of the angular concentration of the source, and thus some control on the modulus of continuity is needed (see the proof of Proposition \ref{PrincipalEstimate} for details.)  Heuristically, it should be reasonable to concede that a sufficiently rough coefficient cannot be well-resolved with any source that is not completely singular.

Fourth, on incoming data concentration: in the case $\Kn \ll 1$, \eqref{HCondition} gives us a condition to guarantee that \eqref{PreSigmaFormula} can be used to reconstruct $\sigma$.  In other words, in the case $\Kn \ll 1$, Theorem \ref{thm:main} gives a quantitative condition on the degree of angular concentration of the source required for tomography to work at all, with possible implications for the RTE-based optical tomography more generally. 

Finally, on stability: this theorem also implies a quantitative stability estimate in terms of $\Kn$ for the acousto-optic reconstruction. In Section 3, we will show that $\frac{f(x-\tau_-)}{\mathcal{A}f(x+\tau_+)}  \sim e^{\frac{c}{\Kn}}$; this means the perturbation in $H$ will be reflected exponentially in the reconstruction of $\sigma$. Applying~\eqref{OldConclusion}, this means that in the regime where \eqref{SigmaFormula} applies, we have
\[
\frac{1}{\Kn}|\sigma_1(x) - \sigma_2(x)| \leq e^{\frac{d\sup \sigma_j}{\Kn}} |H_1(x) - H_2(x)| + \mathcal{O}(h^{\alpha})\,.
\] 
This shows that the stability of the reconstruction decays exponentially as $\Kn \rightarrow 0$, which should be expected from the reliance on singularities in the RTE (see also \cite{LLU2018, ZhaoZ18,LiYangZhong}).

\section{Reconstruction of optical properties}\label{sec:reconstruction}
We show Theorem~\ref{thm:main} in this section. The proof relies on a proper decomposition of the transport solution according to the singularity. This technique, termed the singular decomposition, is the classical technique in the study of optics in inverse problems.

This was introduced in \cite{AniProKov, Bondarenko, CS1, CS2,CS3} and later was applied extensively to treat various settings related to parameter identification of the RTE \cite{Bal18, BalMonard_time_harmonic,CS98, Kuireview, Stefanov_2003,   SU2d,  Wang1999}. For the purpose of this paper, we will decompose the solution to the RTE into two parts - the most singular (ballistic part) and the remainders, and then separate the singular contribution from the rest of $H(x)$.

To start, we first write the solution $u\in L^\infty(\Omega\times\mathbb{S}^{n-1})$ of the equation \eqref{equation RTE} in the form of
\[
u=u_1+u_2\,,
\]
where $u_1$ is the ballistic part that satisfies the X-ray equation and $u_2$ takes care of the remaining terms, namely:
\begin{align}\label{equation_ballistic_remainder}
\begin{cases}
\theta\cdot\nabla u_1=-\frac{\sigma}{\Kn} u_1\,,\\
u_1|_{\Gamma_-}=f\,,\\
\end{cases} \quad\text{and}\hskip1cm
\begin{cases}
\theta\cdot\nabla u_2 = \frac{\sigma_s}{\Kn}  \langle u_2\rangle - \frac{\sigma}{\Kn} u_2 +\frac{\sigma_s}{\Kn} \langle u_1\rangle\,,\\
u_2|_{\Gamma_-}= 0\,.\\
\end{cases}  
\end{align}
The existence of solutions to both equations {follows from Theorem \ref{RTESolutionEstimate}}. Note that the entire boundary condition $f$ is applied on $u_1$ leaving $u_2$ with the trivial boundary condition.

A similar separation can also be performed to the adjoint RTE~\eqref{adjoint equation}. We write $v=v_1+v_2$ with the components satisfying:
\begin{align}\label{equation_ballistic_v}
\begin{cases}
-\theta\cdot\nabla v_1=-\frac{\sigma}{\Kn} v_1\,,\\
v_1|_{\Gamma_+}=g\,,\\
\end{cases} \quad\text{and}\hskip1cm
\begin{cases}
-\theta\cdot\nabla v_2 = \frac{\sigma_s }{\Kn}  \langle v_2\rangle - \frac{\sigma}{\Kn}v_2+\frac{\sigma_s}{\Kn} \langle v_1\rangle\,,\\
v_2|_{\Gamma_+}= 0\,.
\end{cases} 
\end{align}	

With these decompositions, we now rewrite the internal function $H(x)$ as:
\begin{align}\label{internal_data_expand}
H(x) 
&= \frac{\sigma_s(x)}{\Kn} (\langle u_1\rangle\langle v_1\rangle+\langle u_1\rangle\langle v_2\rangle+\langle u_2\rangle\langle v_1\rangle+\langle u_2\rangle\langle v_2\rangle) \notag\\
&\hskip1cm  -\frac{\sigma(x)}{\Kn}(\langle u_1 v_1\rangle + \langle u_1 v_2\rangle + \langle u_2 v_1\rangle + \langle u_2 v_2\rangle)\,,
\end{align}
simply by substituting $u=u_1+u_2$ and $v=v_1+v_2$ into the definition of $H$~\eqref{internal data}.

As described previously, boundary conditions are light beams concentrated in the angle variable $\theta_0\in \mathbb{S}^{n-1}$: 
\begin{equation}\label{eqn:boundary}
   g (x,\theta) = f (x,\theta)
   :=\left\{\begin{array}{ll} 
   c_nh^{\frac{1}{2} (1-n)}  & \hbox{if }|\theta-\theta_0|<h\,,\\
   0 & \hbox{otherwise}\,,\\
   \end{array}\right.
\end{equation}
in $\Omega\times\mathbb{S}^{n-1}$, where $c_n$ is the normalization constant.

In the rest of the section we will show that with such concentrated incoming data \eqref{eqn:boundary} and sufficiently small $h$, the term $\langle u_1 v_1\rangle$ essentially dominates the other terms in $H(x)$. This can be seen from the estimates of the dominating and remaining terms presented in Proposition~\ref{PrincipalEstimate} and Proposition~\ref{ReminderEstimate}, respectively.
Finally, from the explicit dependence of $\langle u_1 v_1\rangle$ on $\sigma$, we can derive the reconstruction of $\sigma$ from the internal data.

\begin{proposition}\label{PrincipalEstimate} 
Suppose $\sigma\in C^{\alpha}(\Om)$, and let $u_i$ and $v_i$, $i=1,2$, satisfy~\eqref{equation_ballistic_remainder} and~\eqref{equation_ballistic_v} respectively with boundary condition defined in~\eqref{eqn:boundary}. Then for each $x$ in $\Omega$, one has
\[
\langle u_1 v_1 \rangle(x) = e^{-\int_{\ell(x,\theta_0)}\frac{\sigma}{\Kn}ds}(1 + \Kn^{-1}\mathcal{O}(h^{\alpha}))\,,
\]
where $\ell(x,\theta)$ denotes the section of the straight line through $x$ in the direction of $\theta$, intersected with $\Omega$.
\end{proposition}

\begin{proof}
The ballistic terms of the solution $u$ to the RTE and the solution $v$ to the adjoint RTE are given explicitly by
\begin{align}\label{BallisticExplicit}
    u_1(x,\theta) =e^{-\int_0^{\tau_-(x,\theta)}{\frac{1}{\Kn}}\sigma(x-s\theta)ds} f(x-\tau_-(x,\theta)\theta,\theta)\,, 
\end{align}
\begin{align}\label{BallisticExplicit v}
    v_1(x,\theta) = e^{-\int_0^{\tau_+(x,\theta)}{\frac{1}{\Kn}}\sigma(x+s\theta)ds} g (x+\tau_+(x,\theta)\theta,\theta)\,,
\end{align}
where $\tau_\pm(x,\theta)$ is the distance from the point $x$ to the boundary in the direction $\pm\theta$. 
 
Substituting these into $\langle u_1 v_1 \rangle(x)$ and using \eqref{eqn:boundary}, we find that for any $x\in \Omega$,
\[
\langle u_1 v_1 \rangle(x) 
= \fint_{\mathbb{S}^{n-1}} e^{-\int_{-\tau_-(x,\theta)}^{\tau_+(x,\theta)}{\frac{1}{\Kn}}\sigma(x+s\theta)ds} c_n^2 h^{1-n}\chi^h_{\theta_0}(\theta) d\theta\,,
\]
where $\chi^h_{\theta_0}(\theta)$ is the characteristic function that is equal to $1$ in $|\theta-\theta_0|<h$ and equal to $0$ otherwise.
The integral in the exponent is precisely the integral over $\ell(x,\theta)$, so we can rewrite this as
\[
\langle u_1 v_1 \rangle(x) 
= c_n^2 h^{1-n}\fint_{\mathbb{S}^{n-1}} e^{-{\frac{1}{ \Kn}}\int_{\ell(x,\theta)}\sigma ds} \chi^h_{\theta_0}(\theta ) d\theta\,.
\]
From \eqref{identity 1} for small $h$, we have
\begin{equation}\label{LebesgueError}
\langle u_1 v_1 \rangle(x) 
= e^{-\frac{1}{\Kn} \int_{\ell(x,\theta_0)}\sigma ds} + R\,,
\end{equation}
where $R$ is the error term
\begin{align*}
R &= \fint_{\mathbb{S}^{n-1}} (e^{-\frac{1}{\Kn} \int_{\ell(x,\theta)}\sigma ds} - e^{-\frac{1}{\Kn} \int_{\ell(x,\theta_0)}\sigma ds})c_n^2 h^{1-n}\chi^h_{\theta_0}(\theta ) d\theta \\
  &=  e^{-\frac{1}{\Kn} \int_{\ell(x,\theta_0)}\sigma ds}\fint_{\mathbb{S}^{n-1}} (e^{-\frac{1}{\Kn} (\int_{\ell(x,\theta)}\sigma ds - \int_{\ell(x,\theta_0)} \sigma ds)} - 1)c_n^2 h^{1-n}\chi^h_{\theta_0}(\theta ) d\theta\,. 
\end{align*}
Since $\sigma \in C^{\alpha}(\Om)$, we have for sufficiently small $h$
\[
R = e^{-\frac{1}{\Kn} \int_{\ell(x,\theta_0)}\sigma ds}\fint_{\mathbb{S}^{n-1}} (e^{-\frac{\mathcal{O}(h^{\alpha})}{\Kn} } - 1)c_n^2 h^{1-n}\chi^h_{\theta_0}(\theta ) d\theta\,. 
\]  
Note that the constant in the $\mathcal{O}$ notation merely depends on the regularity of $\sigma$. We can estimate $|e^{-\frac{\mathcal{O}(h^{\alpha})}{\Kn} } - 1|$ by the Taylor remainder formula if $\frac{\mathcal{O}(h^{\alpha})}{\Kn}$ is small, or by $1$ otherwise. In either case, we get 
\[
R = e^{-\frac{1}{\Kn} \int_{\ell(x,\theta_0)}\sigma ds}\frac{\mathcal{O}(h^{\alpha})}{\Kn}\fint_{\mathbb{S}^{n-1}}c_n^2 h^{1-n}\chi^h_{\theta_0}(\theta )d \theta\,.
\]
Then \eqref{identity 1} implies that
\[
R = \frac{1}{\Kn}e^{-\frac{1}{\Kn} \int_{\ell(x,\theta_0)}\sigma ds}\mathcal{O}(h^{\alpha})\,.
\]
Returning to \eqref{LebesgueError}, we conclude that
\[
\langle u_1 v_1 \rangle(x) 
= e^{-\frac{1}{\Kn}\int_{\ell(x,\theta_0)}\sigma ds} (1 + \Kn^{-1}\mathcal{O}(h^{\alpha}))
\]
as desired.
\end{proof}

The following proposition justifies the smallness of the remainder terms. 
\begin{proposition}\label{ReminderEstimate}
Under the same conditions as in Proposition~\ref{PrincipalEstimate}, suppose $\Kn \leq 1$, there exists a constant $C$ independent of $h$ and $\Kn$ such that
\[
\langle u_j \rangle \langle v_k \rangle (x)\leq C \Kn^{-4}h^{n-1}
\]
for $j,k = 1,2$, and
\[
\langle u_j v_k \rangle (x)\leq C\Kn^{-4}h^{n-1} 
\]
for $j,k = 1,2$ and $j+k > 2$ for each $x$ in $\Omega$.
\end{proposition}

\begin{proof} 
First note that it follows directly from \eqref{BallisticExplicit}, \eqref{eqn:boundary}, and $\sigma>0$,
the angular average of $u_1$ satisfies
	\begin{equation}\label{AverageBallistic}
	\langle u_1 \rangle \leq \mathcal{O}(h^{\frac{n-1}{2}}) \ \hbox{for all }x\in\Omega\,. 
	\end{equation}
Moreover, Theorem \ref{RTESolutionEstimate} and \eqref{equation_ballistic_remainder} tell us that 
	\[
	\|u_2\|_{L^{\infty}(\Omega \times S^{n-1})} \leq  C\Kn^{-2}\left\| \langle u_1 \rangle \right\|_{L^{\infty}(\Omega)}.
	\]
	
Combining with \eqref{AverageBallistic}, we can deduce that
	\begin{equation}\label{u1LinftyEstimate}
	\|u_2\|_{L^{\infty}(\Omega \times S^{n-1})} \leq C\Kn^{-2} \mathcal{O}(h^{\frac{n-1}{2}})\,.
	\end{equation}
	The estimates \eqref{AverageBallistic} and \eqref{u1LinftyEstimate} also hold for $v_1$ and $v_2$, respectively.
	Together these estimates yield all of the bounds of Proposition \ref{ReminderEstimate}.  
	For instance, applying \eqref{AverageBallistic} to $\langle u_1 \rangle$ and \eqref{u1LinftyEstimate} to $\|v_2\|_{L^{\infty}(\Omega \times S^{n-1})}$ we get 
	\[
	\langle u_1 v_2 \rangle \leq \langle u_1 \rangle \|v_2\|_{L^{\infty}(\Omega \times S^{n-1})}\leq \Kn^{-2}\mathcal{O}(h^{n-1})\,.
	\]
	Similarly, we can derive
	\[
	\langle u_2 v_2 \rangle \leq C\Kn^{-4} \mathcal{O}(h^{n-1})\,,
	\]
	where $C$ is a constant independent of $\Kn$ and $h$.
\end{proof}

The main result is a natural consequence of the previous two propositions.
\begin{proof}[Proof of Theorem~\ref{thm:main}]
We first rewrite~\eqref{internal_data_expand} as
\begin{align*}
 \, & H(x) + \frac{\sigma(x)}{\Kn}\langle u_1 v_1\rangle \\
&= \frac{\sigma_s(x)}{\Kn} (\langle u_1\rangle\langle v_1\rangle+\langle u_1\rangle\langle v_2\rangle+\langle u_2\rangle\langle v_1\rangle+\langle u_2\rangle\langle v_2\rangle) \notag\\
&\hskip1cm  -\frac{\sigma(x)}{\Kn}(\langle u_1 v_2\rangle + \langle u_1 v_2\rangle + \langle u_2 v_2\rangle)\,.
\end{align*}
Applying Proposition~\ref{PrincipalEstimate} to the $\langle u_1 v_1\rangle$ term and Proposition~\ref{ReminderEstimate} to everything else, for $\Kn\ll 1$, we get
\[
H(x) + \frac{\sigma(x)}{\Kn}\left(e^{-{\frac{1}{\Kn}}\int_{\ell(x,\theta_0)}\sigma ds}(1 + \Kn^{-1}\mathcal{O}(h^{\alpha}))\right) =  \Kn^{-5}\mathcal{O}(h^{n-1})\,.
\]
Multiplying through by $e^{\int_{\ell(x,\theta_0)}{\frac{\sigma}{\Kn}}ds} \,$, one has
\[
\frac{\sigma(x)}{\Kn} + e^{\int_{\ell(x,\theta_0)}{\frac{\sigma}{\Kn}}ds} H(x) = \frac{1}{\Kn^2}\mathcal{O}(h^{\alpha}) + \frac{1}{\Kn^5}\mathcal{O}(h^{n-1})e^{\frac{d \sup \sigma }{\Kn}}\,.
\]
In the zero limit of $\Kn$, the exponential term dominates, and the equation becomes 
\begin{equation}\label{OldConclusion}
\frac{\sigma(x)}{\Kn} + e^{\int_{\ell(x,\theta_0)}{\frac{\sigma}{\Kn}}ds} H(x) \leq \mathcal{O}(h^{\alpha})e^{\frac{2d \sup \sigma }{\Kn}}\,.
\end{equation}
To estimate the exponential factor in front of $H(x)$, we recall that from \eqref{BallisticExplicit}, $u_1$ has explicit formula:
\[
u_1(x+\tau_+(x,\theta_0)\theta_0, \theta_0) =e^{-{\frac{1}{\Kn}}\int_{\ell(x,\theta_0)}\sigma ds} f(x-\tau_-(x,\theta_0)\theta_0,\theta_0)\,,
\]
and from~\eqref{u1LinftyEstimate} we also have:
\[
|u_2(x+\tau_+(x,\theta_0)\theta_0, \theta_0)| \leq \Kn^{-2} \mathcal{O}(h^{\frac{n-1}{2}})\,.
\]
Considering the definition of the albedo operator:
\[
\mathcal{A}(f)(x+\tau_+(x,\theta_0)\theta_0, \theta_0) = u(x+\tau_+(x,\theta_0)\theta_0, \theta_0) = u_1(x+\tau_+(x,\theta_0)\theta_0, \theta_0) + u_2(x+\tau_+(x,\theta_0)\theta_0, \theta_0)\,,
\]
we finally have
\begin{align}\label{TotalAttenuation}
\frac{\mathcal{A}(f)(x+\tau_+(x,\theta_0)\theta_0, \theta_0)}{f(x-\tau_-(x,\theta_0)\theta_0,\theta_0)} &= e^{-{\frac{1}{\Kn}}\int_{\ell(x,\theta_0)}\sigma ds} + \frac{\mathcal{O}(h^{\frac{n-1}{2}})}{\Kn^{2}f(x-\tau_-(x,\theta_0)\theta_0,\theta_0)}\nonumber\\
& =e^{-{\frac{1}{\Kn}}\int_{\ell(x,\theta_0)}\sigma ds} + \frac{\mathcal{O}(h^{n-1})}{\Kn^{2}}\,,
\end{align}
where we also used the fact that $f|_{\Gamma_-}$ defined in \eqref{eqn:boundary}. 
By substituting this back to~\eqref{OldConclusion}, it immediately suggests:
\[
 \frac{\sigma(x)}{\Kn} + {\color{black}\frac{f(x-\tau_-(x,\theta_0)\theta_0,\theta_0)}{\mathcal{A}(f)(x+\tau_+(x,\theta_0)\theta_0, \theta_0)} }H(x) = \mathcal{O}(h^{\alpha})e^{\frac{2d \sup \sigma }{\Kn}} + \frac{\mathcal{O}(h^{n-1})}{\Kn^{2}}H(x)\,.
\]
Since~\eqref{OldConclusion} also implies that $H(x)$ is of order $e^{\frac{2d \sup \sigma }{\Kn}}$, we conclude
\[
\frac{\sigma(x)}{\Kn} +{\color{black} \frac{f(x-\tau_-(x,\theta_0)\theta_0,\theta_0)}{\mathcal{A}(f)(x+\tau_+(x,\theta_0)\theta_0, \theta_0)} } H(x) =  \mathcal{O}(h^{\alpha}) e^{\frac{3d \sup \sigma}{\Kn}}\,,
\]
Thus, Theorem \ref{thm:main} now follows. 
\end{proof}

\section{Numeric Examples}\label{sec:numeric}
In this section we provide numerical evidence that verifies the analysis above.

As a numerical setup, we let $\Omega = [0,1]^2$, and mesh grids are sampled to resolve small scales in $\Kn$. Since we only concern ourselves with the behavior of the different terms in $H(x)$ in terms of $\Kn$ and $h$, we use the simple media $\sigma = 1$ all over the domain. In particular, the incoming light is concentrated at $(x=0,y=1/2)$ with $h$ indicating the concentration in velocity domain. In Figure~\ref{fig:u1_ave} we plot the velocity average of the ballistic part, $\langle u_1\rangle$, for different $\Kn$ with $h=\frac{2\pi}{25}$ and $h=\frac{14\pi}{25}$, and it can be clearly seen that smaller $h$ narrows down the spreading of $\langle u_1\rangle$ while smaller $\Kn$ gives stronger decay of $\langle u_1\rangle$ in space. Moreover, we also plot $\langle u_2\rangle$ in Figure~\ref{fig:u2_ave}. Similar to the case of $\langle u_1\rangle$, as $h$ increases the function is more spread out, while smaller $\Kn$ leads to the stronger decay.

We also plot $\langle u_1v_1\rangle$ as a function of $\Kn$ and $h$. 
As seen in Figure~\ref{fig:u1v1}, the quantity decays exponentially fast with respect to small $\Kn$ and algebraically fast with respect to small $h$,  which confirms with the result in Proposition~\ref{PrincipalEstimate}. In addition, the remainder term $H(x)-\frac{1}{\Kn}\langle u_1v_1\rangle$ is plotted in Figure~\ref{fig:rest}.

Lastly, the quantitative dependence of error on $\Kn$ and $h$ is plotted in Figure~\ref{fig:error} and the observation confirms with the prediction of our theorem. Specifically, this error is an exponential function of $\frac{1}{\Kn}$ and grows algebraically with respect to $h$, where the error is defined to be the relative error, namely:  
\[
\text{Error}(x) = \frac{\Kn H(x)-\langle u_1v_1\rangle}{\langle u_1v_1\rangle}
\]
and the plot shows the $L^2$ norm over the space $x\in\Omega$.

\begin{figure}[ht]
\includegraphics[height = 0.5\textheight,width = 1.2\textwidth,angle = 90]{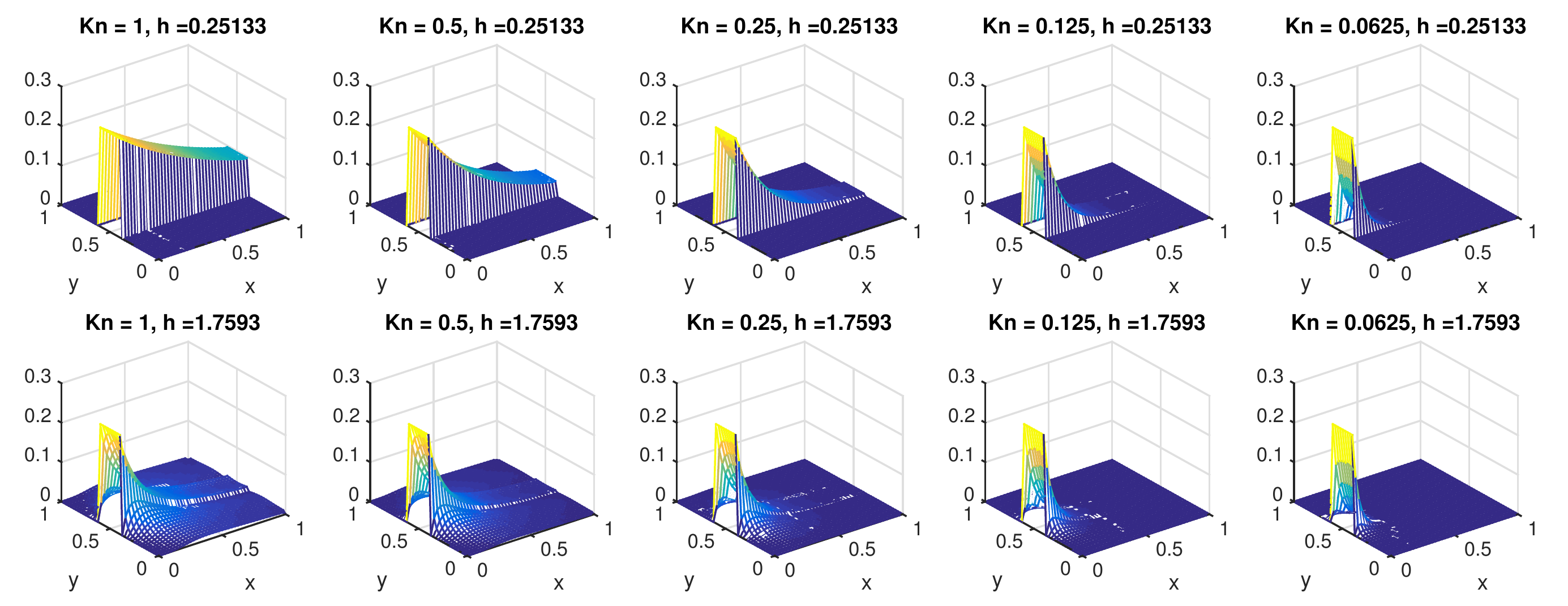}
\caption{The plot show $\langle u_1\rangle$ with different $\Kn$ and $h$. From left to right, $\Kn$ decreases as $2^{-k}$ with $k=0,1,2,3,4$. The two rows are for 
$h$ being $\frac{2\pi}{25}$ and $\frac{14\pi}{25}$ respectively.}\label{fig:u1_ave}
\end{figure}

\begin{figure}[ht]
\includegraphics[height = 0.5\textheight,width = 1.2\textwidth,angle = 90]{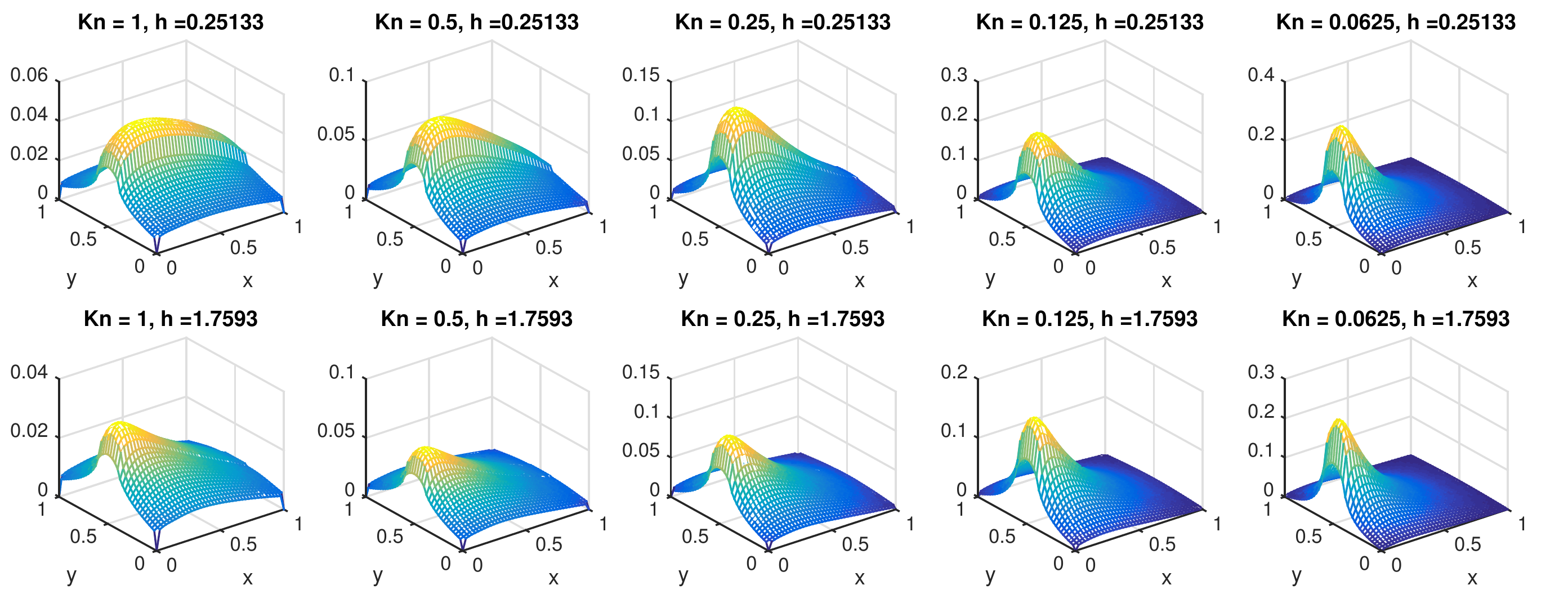}
\caption{The plot show $\langle u_2\rangle$ with different $\Kn$ and $h$. From left to right, $\Kn$ decreases as $2^{-k}$ with $k=0,1,2,3,4$. The two rows are for 
$h$ being $\frac{2\pi}{25}$ and $\frac{14\pi}{25}$ respectively.}\label{fig:u2_ave}
\end{figure}

\begin{figure}[ht]
\includegraphics[height = 0.5\textheight,width = 1.2\textwidth,angle = 90]{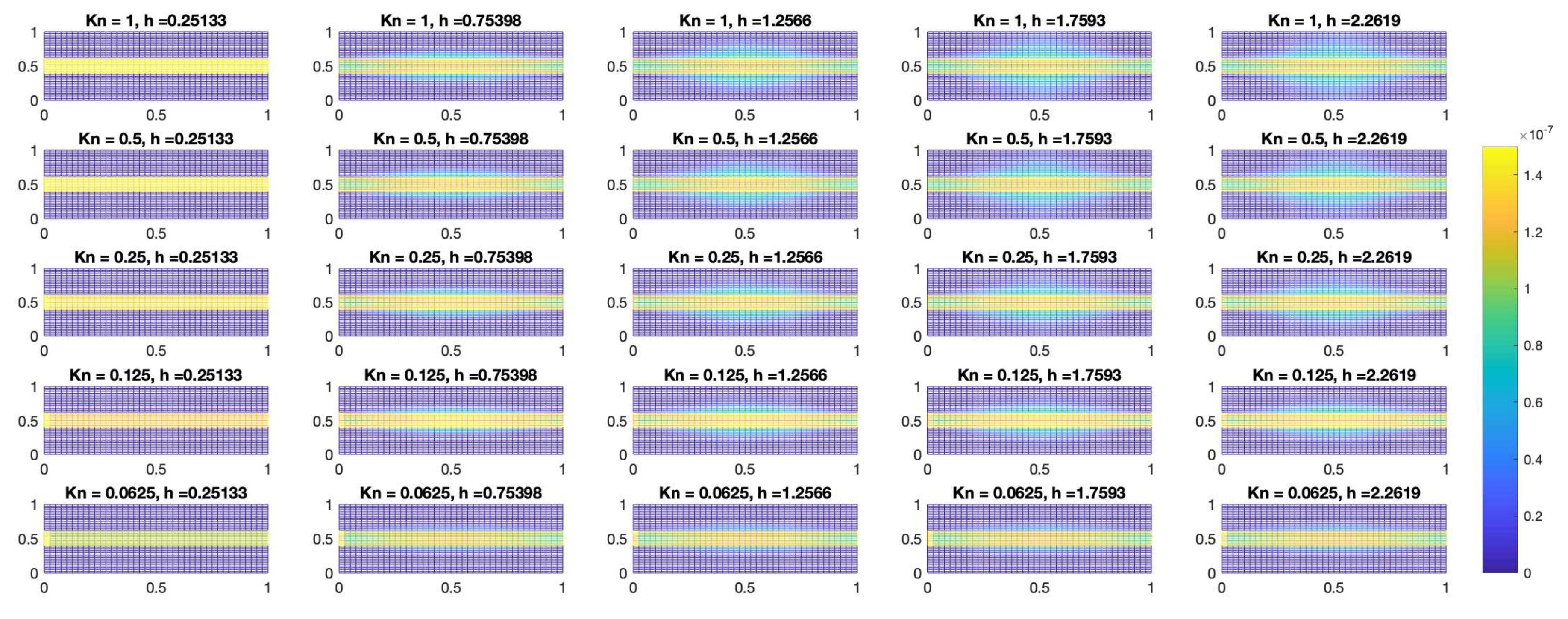}
\caption{As $h$ increases, the intensity of $\langle u_1v_1\rangle$ is more spread out, and as $\Kn$ decreases, the intensity quickly decays and is low in the interior of the domain.}\label{fig:u1v1}
\end{figure}

\begin{figure}[ht]
\includegraphics[height = 0.5\textheight,width = 1.2\textwidth,angle = 90]{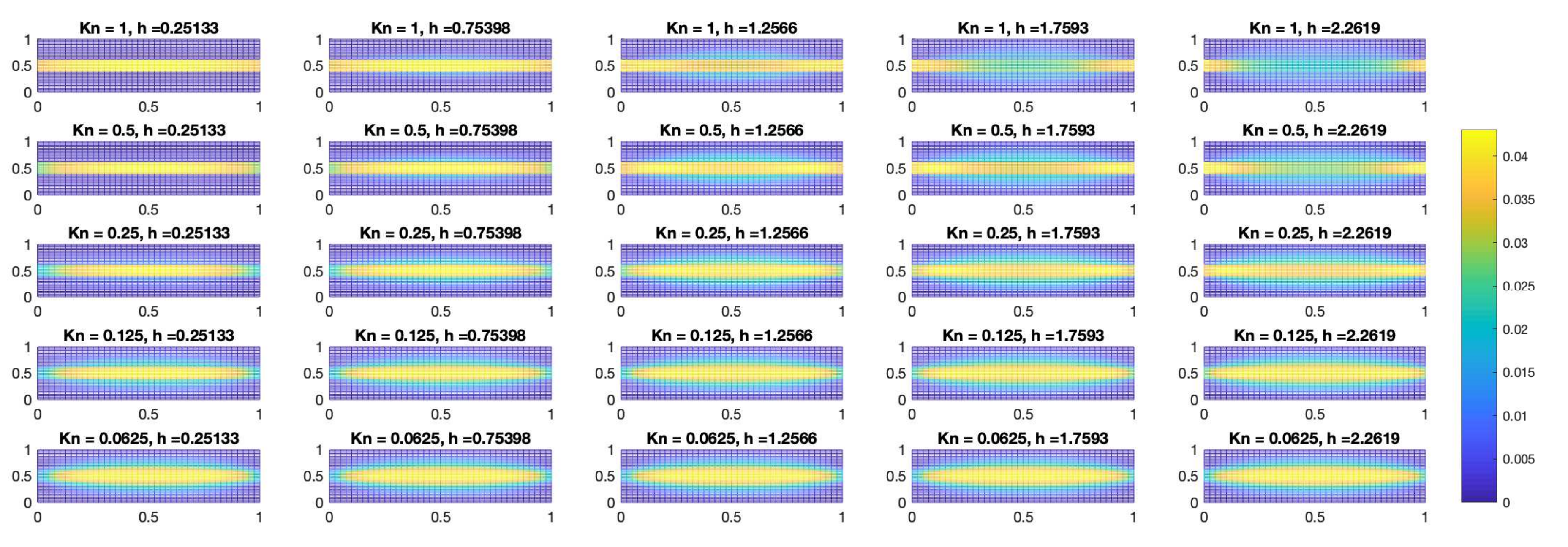}
\caption{The plot shows $H(x)-\langle u_1v_1\rangle$ as a function of $h$ and $\Kn$.}\label{fig:rest}
\end{figure}

\begin{figure}[ht]
\includegraphics[width = 0.45\textwidth]{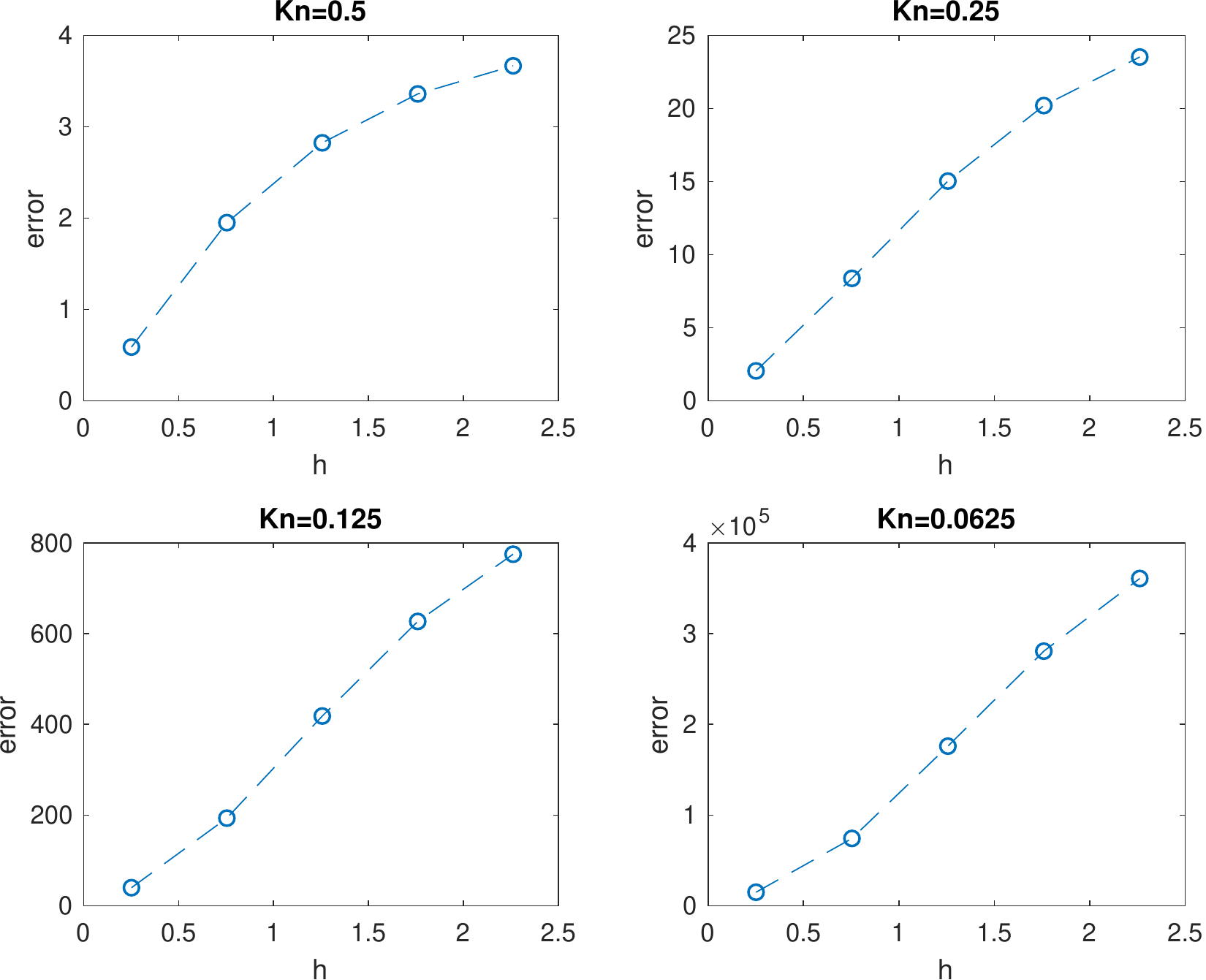}\hspace{0.3cm}
\includegraphics[width = 0.45\textwidth]{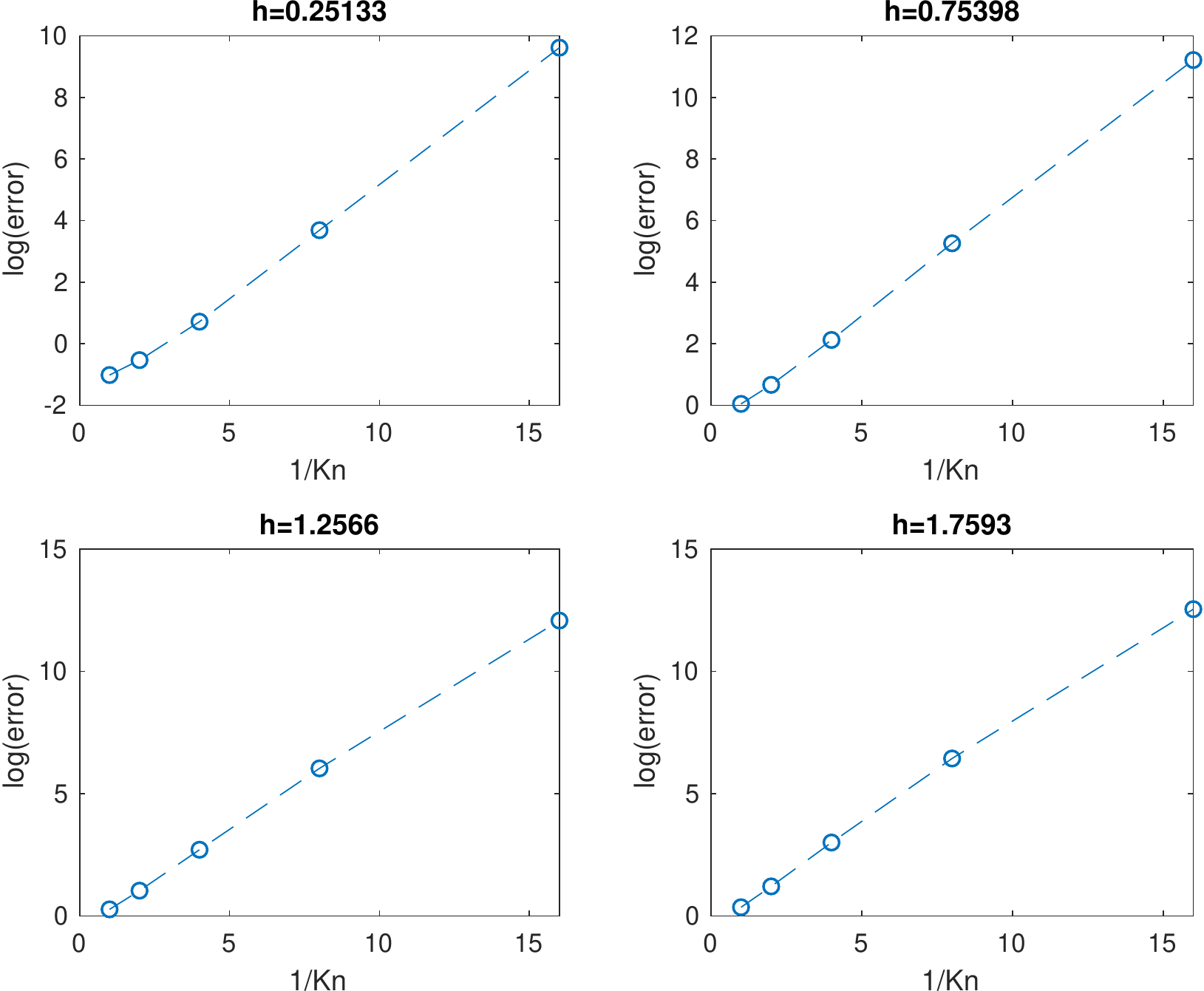}
\caption{These two plots show the error as a function of $\Kn$ and $h$. For every fixed $\Kn$, the error grows almost linearly with respect to $h$, and for every fixed $h$, the logarithmic error grows almost linearly with respect to $\frac{1}{\Kn}$.}\label{fig:error}
\end{figure}

\vspace{1cm}

\bibliographystyle{abbrv}

\bibliography{transbib}

\begin{thebibliography}{10}

\bibitem{A1}
G.~Alessandrini.
\newblock Stable determination of conductivity by boundary measurements.
\newblock {\em Appl. Anal.}, 27:153--172, 1988.

\bibitem{AmmariBGNS}
H.~Ammari, E.~Bossy, J.~Garnier, L.~H. Nguyen, and L.~Seppecher.
\newblock A reconstruction algorithm for ultrasound-modulated diffuse optical
  tomography.
\newblock {\em Proc. Amer. Math. Soc.}, 142:3221--3236, 2014.

\bibitem{AmmariGNS}
H.~Ammari, J.~Garnier, L.~H. Nguyen, and L.~Seppecher.
\newblock Reconstruction of a piecewise smooth absorption coefficient by an
  acousto-optic process.
\newblock {\em Commun. Part. Diff. Eq.}, 38:1737--1762, 2013.

\bibitem{AmmariNS}
H.~Ammari, L.~H. Nguyen, and L.~Seppecher.
\newblock Reconstruction and stability in acousto-optic imaging for absorption
  maps with bounded variation.
\newblock {\em J. Funct. Anal.}, 267:4361--4398, 2014.

\bibitem{AniProKov}
D.~Anikonov, I.~Prokhorov, and A.~Kovtanyuk.
\newblock Investigation of scattering and absorbing media by the methods of
  {X}-ray tomography.
\newblock {\em Journal of Inverse and Ill-posed Problems}, 1:259--281, 1993.

\bibitem{Arridge99}
S.~Arridge.
\newblock Optical tomography in medical imaging.
\newblock {\em Inverse Problems}, 15:R41--93, 1999.

\bibitem{Arridge_Schotland09}
S.~R. Arridge and J.~C. Schotland.
\newblock Optical tomography: forward and inverse problems.
\newblock {\em Inverse Problems}, 25(12):123010, 2009.

\bibitem{Bal2012review}
G.~Bal.
\newblock {\em Hybrid inverse problems and internal functionals, Inside Out II,
  MSRI Publications, G. Uhlmann Editor}.
\newblock Cambridge University Press, 2012.

\bibitem{BalChungSchotland}
G.~Bal, F.~J. Chung, and J.~C. Schotland.
\newblock Ultrasound modulated bioluminescence tomography and controllability
  of the radiative transport equation.
\newblock {\em SIAM J. Math. Analysis}, 48(2):1332--1347, 2016.

\bibitem{Bal09}
G.~Bal and A.~Jollivet.
\newblock Time-dependent angularly averaged inverse transport.
\newblock {\em Inverse Problems}, 25(7):075010, 2009.

\bibitem{Bal10a}
G.~Bal and A.~Jollivet.
\newblock Stability for time-dependent inverse transport.
\newblock {\em SIAM J. Math. Anal.}, 42(2):679--700, 2010.

\bibitem{Bal18}
G.~Bal and A.~Jollivet.
\newblock Generalized stability estimates in inverse transport theory.
\newblock {\em Inverse problems and Imaging}, 12(1):59--90, 2018.

\bibitem{BalMonard_time_harmonic}
G.~Bal and F.~Monard.
\newblock Inverse transport with isotropic time-harmonic sources.
\newblock {\em SIAM Journal on Mathematical Analysis}, 44(1):134--161, 2012.

\bibitem{BalMoskow2014}
G.~Bal and S.~Moskow.
\newblock Local inversions in ultrasound modulated optical tomography.
\newblock {\em Inverse Problems}, 30(2):025005, 2014.

\bibitem{BalSchotland2010}
G.~Bal and J.~C. Schotland.
\newblock Inverse scattering and acousto-optic imaging.
\newblock {\em Phys. Rev. Letters}, 104:043902, 2010.

\bibitem{BalSchotland2014}
G.~Bal and J.~C. Schotland.
\newblock Ultrasound-modulated bioluminescence tomography.
\newblock {\em Phys. Rev. E}, 89:031201, 2014.

\bibitem{Bondarenko}
A.~Bondarenko.
\newblock Singular structure of the fundamental solution of the transport
  equation, and inverse problems in particle scattering theory.
\newblock {\em Dokl. Akad. Nauk SSSR}, 322:274--276, 1992.

\bibitem{Chen_2018}
K.~Chen, Q.~Li, and J.-G. Liu.
\newblock Online learning in optical tomography: a stochastic approach.
\newblock {\em Inverse Problems}, 34(7):075010, may 2018.

\bibitem{CLW}
K.~Chen, Q.~Li, and L.~Wang.
\newblock Stability of stationary inverse transport equation in diffusion
  scaling.
\newblock {\em Inverse Problems}, 34(2), 2018.

\bibitem{Cheng_Gamba_Ren_Doping}
Y.~Cheng, I.~M. Gamba, and K.~Ren.
\newblock Recovering doping profiles in semiconductor devices with the
  {B}oltzmann-{P}oisson model.
\newblock {\em J. Comput. Phys.}, 230(9):3391--3412, May 2011.

\bibitem{CS1}
M.~Choulli and P.~Stefanov.
\newblock Scattering inverse pour l'\'equation du transport et relations entre
  les op\'erateurs de scattering et d'alb\'edo.
\newblock {\em C. R. Acad. Sci. Paris}, 320:947--952, 1995.

\bibitem{CS2}
M.~Choulli and P.~Stefanov.
\newblock Inverse scattering and inverse boundary value problems for the linear
  {B}oltzmann equation.
\newblock {\em Comm. P.D.E.}, 21:763--785, 1996.

\bibitem{CS3}
M.~Choulli and P.~Stefanov.
\newblock Reconstruction of the coefficients of the stationary transport
  equation from boundary measurements.
\newblock {\em Inverse Problems}, 12:L19--L23, 1996.

\bibitem{CS98}
M.~Choulli and P.~Stefanov.
\newblock An inverse boundary value problem for the stationary transport
  equation.
\newblock {\em Osaka J. Math.}, 36:87--104, 1998.

\bibitem{ChungJS}
F.~Chung, J.~Hoskins, and J.~Schotland.
\newblock Coherent acousto-optic tomography with diffuse light.
\newblock {\em preprint}, 2018.

\bibitem{CHS2019}
F.~Chung, J.~Hoskins, and J.~Schotland.
\newblock A transport model for multi-frequency acousto-optic tomography.
\newblock {\em arXiv:1910.04798}, 2019.

\bibitem{ChungSchotland}
F.~Chung and J.~Schotland.
\newblock Inverse transport and acousto-optic imaging.
\newblock {\em SIAM J. Math. Analysis}, 49:4704--4721, 2017.

\bibitem{Elson2011}
D.~S. Elson, R.~Li, C.~Dunsby, R.~Eckersley, and M.-X. Tang.
\newblock Ultrasound-mediated optical tomography: a review of current methods.
\newblock {\em Interface Focus}, 1:632--648, 2011.

\bibitem{jin2019convergence}
B.~Jin, Z.~Zhou, and J.~Zou.
\newblock On the convergence of stochastic gradient descent for nonlinear
  ill-posed problems.
\newblock {\em arXiv:1907.03132}, 2019.

\bibitem{Kuchment}
P.~Kuchment.
\newblock {\em Mathematics of Hybrid Imaging: A Brief Review. In: Sabadini I.,
  Struppa D. (eds)}, volume~16.
\newblock Springer Proceedings in Mathematics, 2012.

\bibitem{LLU2018}
R.-Y. Lai, Q.~Li, and G.~Uhlmann.
\newblock Inverse problems for the stationary transport equation in the
  diffusion scaling.
\newblock {\em SIAM Journal on Applied Mathematics}, 79(6):2340--2358, 2019.

\bibitem{Lesaffrea}
M.~Lesaffrea, F.~Jeana, A.~Funkea, P.~Santosa, M.~Atlana, B.~Forgeta,
  E.~Bossya, F.~Ramaza, A.~Boccaraa, M.~Grossb, P.~Delayec, and G.~Roosenc.
\newblock Acousto-optic imaging techniques for optical diagnosis.
\newblock {\em IFAC Proceedings Volumes}, 39(18):11--15, 2006.

\bibitem{LiYangZhong}
W.~Li, Y.~Yang, and Y.~Zhong.
\newblock A hybrid inverse problem in the fluorescence ultrasound modulated
  optical tomography in the diffusive regime.
\newblock {\em SIAM Journal on Applied Mathematics}, 79(1):356--376, 2019.

\bibitem{Hielscher_rheumatoid2}
L.~D. Montejo, J.~Jia, H.~K. Kim, U.~J. Netz, S.~Blaschke, G.~A. Muller, and
  A.~H. Hielscher.
\newblock Computer-aided diagnosis of rheumatoid arthritis with optical
  tomography, part 2: image classification.
\newblock {\em Journal of Biomedical Optics}, 18(7):076002--076002, 2013.

\bibitem{Kuireview}
K.~Ren.
\newblock Recent developments in numerical techniques for transport-based
  medical imaging methods.
\newblock {\em Commun. Comput. Phys.}, 8(1):1--50, 2010.

\bibitem{Ren_fPAT}
K.~Ren, R.~Zhang, and Y.~Zhong.
\newblock Inverse transport problems in quantitative pat for molecular imaging.
\newblock {\em Inverse Problems}, 31(12):125012, 2015.

\bibitem{Stefanov_2003}
P.~Stefanov.
\newblock {\em Inverse problems in transport theory}, volume~47.
\newblock Inside Out: Inverse Problems; MSRI Publications, edited by G.
  Uhlmann, 2003.

\bibitem{SU2d}
P.~Stefanov and G.~Uhlmann.
\newblock Optical tomography in two dimensions.
\newblock {\em Methods Appl. Anal.}, 10:1--9, 2003.

\bibitem{Arridge_diffusion}
T.~Tarvainen, B.~T. Cox, J.~P. Kaipio, and S.~R. Arridge.
\newblock Reconstructing absorption and scattering distributions in
  quantitative photoacoustic tomography.
\newblock {\em Inverse Problems}, 28(8):084009, 2012.

\bibitem{Arridge_couple}
T.~Tarvainen, V.~Kolehmainen, S.~R. Arridge, and J.~P. Kaipio.
\newblock Image reconstruction in diffuse optical tomography using the coupled
  radiative transport–diffusion model.
\newblock {\em Journal of Quantitative Spectroscopy and Radiative Transfer},
  112(16):2600 -- 2608, 2011.

\bibitem{Arridge_diffusion_error}
T.~Tarvainen, V.~Kolehmainen, A.~Pulkkinen, M.~Vauhkonen, M.~Schweiger, S.~R.
  Arridge, and J.~P. Kaipio.
\newblock An approximation error approach for compensating for modelling errors
  between the radiative transfer equation and the diffusion approximation in
  diffuse optical tomography.
\newblock {\em Inverse Problems}, 26(1):015005, 2010.

\bibitem{uhlmann2009electrical}
G.~Uhlmann.
\newblock Electrical impedance tomography and {C}alder{\'o}n's problem.
\newblock {\em Inverse problems}, 25(12):123011, 2009.

\bibitem{Wang1999}
J.-N. Wang.
\newblock Stability estimates of an inverse problem for the stationary
  transport equation.
\newblock {\em Ann. Inst. H. Poincar\'e Phys. Th\'eor.}, 70(5):473--495, 1999.

\bibitem{LVWang}
L.~Wang.
\newblock Ultrasound-mediated biophotonic imaging: A review of acousto-optical
  tomography and photo-acoustic tomography.
\newblock {\em Disease Markers}, 19:123--138, 2003, 2004.

\bibitem{ZhaoZ18}
H.~Zhao and Y.~Zhong.
\newblock Instability of an inverse problem for the stationary radiative
  transport near the diffusion limit.
\newblock {\em arXiv:1809.01790}, 2018.

\end{thebibliography}

\end{document}